\documentclass[12pt,reqno]{amsart}
\calclayout
\usepackage{amsmath}
\usepackage[colorlinks=true,linkcolor=blue,citecolor=blue]{hyperref}
\usepackage{amssymb,amsthm,xypic,stmaryrd,graphicx,mathtools}
\usepackage[all]{xy}
\usepackage{bbm, dsfont}
\usepackage{boxedminipage,xcolor}
\usepackage[shortlabels]{enumitem}
\usepackage{thmtools}
\usepackage{thm-restate}
\DeclareMathOperator{\supp}{supp}

\DeclareMathOperator{\alg}{alg}
\DeclareMathOperator{\red}{red}

\DeclareMathOperator{\Av}{Av}

\DeclareMathOperator{\ind}{index}

\DeclareMathOperator{\End}{End}

\DeclareMathOperator{\Spin}{Spin}

\DeclareMathOperator{\restr}{restr}

\newcommand{\beq}[1]{\begin{equation} \label{#1}}
\newcommand{\eeq}{\end{equation}}

\newcommand{\bea}{\begin{eqnarray}}
\newcommand{\eea}{\end{eqnarray}}

\usepackage{hyperref}
\usepackage{cleveref}
\usepackage{fancyhdr}

\begin{document}
	\pagestyle{plain}
	\theoremstyle{plain}
	\newtheorem{theorem}{Theorem}[section]
	\newtheorem{thm}{Theorem}[section]
	\newtheorem{lemma}[theorem]{Lemma}
	\newtheorem{proposition}[theorem]{Proposition}
	\newtheorem{assumption}[theorem]{Assumption}
	\newtheorem{prop}[theorem]{Proposition}
	\newtheorem{corollary}[theorem]{Corollary}
	\newtheorem{conjecture}[theorem]{Conjecture}
	
	\theoremstyle{definition}
	\newtheorem{definition}[theorem]{Definition}
	\newtheorem{defn}[theorem]{Definition}
	\newtheorem{example}[theorem]{Example}
	\newtheorem{remark}[theorem]{Remark}
	\newtheorem{rem}[theorem]{Remark}
	\newtheoremstyle{note}
	{}
	{}
	{\itshape}
	{}
	{\bfseries}
	{.}
	{ }
	{}
	\theoremstyle{note}
	\newtheorem*{theoremmain}{Theorem \ref{thm:main}}
	
	\newcommand{\C}{\mathbb{C}}
	\newcommand{\R}{\mathbb{R}}
	\newcommand{\Z}{\mathbb{Z}}
	\newcommand{\N}{\mathbb{N}}
	
	\newcommand{\Supp}{{\rm Supp}}
	
	\newcommand{\field}[1]{\mathbb{#1}}
	\newcommand{\bZ}{\field{Z}}
	\newcommand{\bR}{\field{R}}
	\newcommand{\bC}{\field{C}}
	\newcommand{\bN}{\field{N}}
	\newcommand{\bT}{\field{T}}
	
	\newcommand{\cB}{{\mathcal{B} }}
	\newcommand{\cK}{{\mathcal{K} }}
	\newcommand{\cF}{{\mathcal{F} }}
	\newcommand{\cO}{{\mathcal{O} }}
	\newcommand{\cE}{\mathcal{E}}
	\newcommand{\cS}{\mathcal{S}}
	\newcommand{\calL}{\mathcal{L}}
	
	\newcommand{\HH}{{\mathcal{H} }}
	\newcommand{\tilH}{\widetilde{\HH}}
	\newcommand{\HX}{\HH_X}
	\newcommand{\Hpi}{\HH_{\pi}}
	\newcommand{\HHpi}{\HH \otimes \HH_{\pi}}
	\newcommand{\Ltwopi}{L^2_{\pi}(X, \HHpi)}
	
	\newcommand{\KK}{K\!K}

	\newcommand{\D}{D \hspace{-0.27cm }\slash}
	\newcommand{\Dsmall}{D \hspace{-0.19cm }\slash}
	
	\newcommand{\mybigwedge}{\textstyle{\bigwedge}}
	\newenvironment{proofof}[1]
	{\noindent \emph{Proof of #1.}}{\hfill $\square$}

	\newcommand{\CGmax}{C^*_{G, \max}}
	\newcommand{\DGmax}{D^*_{G, \max}}
	\newcommand{\CGred}{C^*_{G, \red}}
	\newcommand{\CGalg}{C^*_{G, \alg}}
	\newcommand{\DGalg}{D^*_{G, \alg}}
	\newcommand{\CGker}{C^*_{G, \ker}}
	\newcommand{\Cmax}{C^*_{\max}}
	\newcommand{\Dmax}{D^*_{\max}}
	\newcommand{\Cred}{C^*_{\red}}
	\newcommand{\Calg}{C^*_{\alg}}
	\newcommand{\Dalg}{D^*_{\alg}}
	\newcommand{\Cker}{C^*_{\ker}}
	\newcommand{\tilCalg}{\widetilde{C}^*_{\alg}}
	\newcommand{\Cpiker}{C^*_{\pi, \ker}}
	\newcommand{\Cpialg}{C^*_{\pi, \alg}}
	\newcommand{\Cpimax}{C^*_{\pi, \max}}
	
	\newcommand{\Avpi}{\Av^{\pi}}
	
	\newcommand{\Gtc}{\Gamma^{\infty}_{tc}}
	\newcommand{\tilD}{\widetilde{D}}
	\newcommand{\XoneH}{X^{\HH}_1}
	
	\newcommand{\XX}{\mathfrak{X}}
	
	\def\kt{\mathfrak{t}}
	\def\kk{\mathfrak{k}}
	\def\kp{\mathfrak{p}}
	\def\kg{\mathfrak{g}}
	\def\kh{\mathfrak{h}}

	\newcommand{\pilamrho}{[\pi_{\lambda+\rho}]}
	
	\newcommand{\Trestr}{\mathcal{T}_{\restr}}
	\newcommand{\TdN}{\mathcal{T}_{d_N}}
	
	\newcommand{\omG}{\om/\hspace{-1mm}/G}
	\newcommand{\om}{\omega} \newcommand{\Om}{\Omega}
	
	\newcommand{\QcwR}{quantization commutes with reduction}
	
	\newcommand{\Spinc}{\Spin^c}
	
	\def\kt{\mathfrak{t}}
	\def\kk{\mathfrak{k}}
	\def\kp{\mathfrak{p}}
	\def\kg{\mathfrak{g}}
	\def\kh{\mathfrak{h}}
	
	\newcommand{\ddt}{\left. \frac{d}{dt}\right|_{t=0}}
	\newcommand{\norm}[1]{\left\lVert#1\right\rVert}
	\newcommand{\normbig}[1]{\big\lVert#1\big\rVert}
	
	\newcommand{\Todo}{\color{red}\textbf{TO DO}\color{black}}
	
	\title{A Lichnerowicz Vanishing Theorem for the Maximal Roe Algebra}

\author{Hao Guo}
\address[Hao Guo]{ Department of Mathematics, Texas A\&M University }
\email{haoguo@math.tamu.edu}
\thanks{The first author is partially supported by NSF 2000082.}

\author{Zhizhang Xie}
\address[Zhizhang Xie]{ Department of Mathematics, Texas A\&M University }
\email{xie@math.tamu.edu}
\thanks{The second author is partially supported by NSF 1800737 and NSF 1952693.}

\author{Guoliang Yu}
\address[Guoliang Yu]{ Department of
	Mathematics, Texas A\&M University}
\email{guoliangyu@math.tamu.edu}
\thanks{The third author is partially supported by NSF 1700021, NSF 2000082, and the Simons Fellows Program.	}

\subjclass[2010]{46L80, 58B34, 53C20}

	\date{}
	\maketitle
	\begin{abstract}
	We show that if a countable discrete group acts properly and isometrically on a spin manifold of bounded Riemannian geometry and uniformly positive scalar curvature, then, under a suitable condition on the group action, the maximal higher index of the Dirac operator vanishes in $K$-theory of the maximal equivariant Roe algebra. The group action is 
	\emph{not} assumed to be cocompact. A key step in the proof is to establish a functional calculus for the Dirac operator in the maximal equivariant 
\emph{uniform} Roe algebra. This allows us to prove vanishing of the index of the Dirac operator in $K$-theory of this algebra, which in turn yields the result for the maximal higher index.

	\end{abstract}
	\tableofcontents
	\section{Introduction}
	\label{sec:intro}
	The connection between index theory and the existence question for metrics of positive scalar curvature on spin manifolds goes back to the classical version of the Lichnerowicz vanishing theorem \cite{Lichnerowicz}, which states that if a closed spin manifold admits a metric of positive scalar curvature, then the Fredholm index of the Dirac operator vanishes. By taking into account the fundamental group, one can define a more refined invariant called the \emph{higher index}. For cocompact group actions, this index takes values in the $K$-theory of the (reduced or maximal) group $C^*$-algebra. By a pioneering result of Rosenberg \cite{Rosenberg3}, if a closed spin manifold admits a positive scalar curvature metric, then the higher index of the Dirac operator lifted to the universal cover vanishes.
					
	For group actions that are not necessarily cocompact, the higher index of the Dirac operator takes values in the $K$-theory of the equivariant Roe algebra \cite{Roetopology}, which for cocompact manifolds is isomorphic to the $K$-theory of the \emph{reduced} group $C^*$-algebra. In \cite{GWY}, Gong, Wang, and Yu introduced a version of this index that takes values in the $K$-theory of the \emph{maximal} equivariant Roe algebra, in analogy to the situation for group $C^*$-algebras. The maximal equivariant Roe algebra is well-defined when the manifold and group action satisfy certain geometric conditions that we shall make precise. One of the advantages of working with the maximal version of the higher index is that it enjoys better functoriality properties than its reduced counterpart, making its $K$-groups more computable in some cases \cite{Xiaoman, Oyono}. The maximal higher index is used in an essential way in the recent work \cite{CWY} of Chang, Weinberger, and Yu, who establish a new index theory for non-compact manifolds and use it to provide examples of manifolds with exotic scalar curvature behavior.
	
	A basic notion on which the results of \cite{CWY} rest is that the maximal higher index of the spin-Dirac operator vanishes in the presence of uniformly positive scalar curvature. However, it was pointed out in \cite{SchickStudent} that, compared to the reduced setting, a Lichnerowicz-type argument for such a vanishing result needs to be carried out with more care, due to analytical difficulties that arise in connection with the maximal equivariant Roe algebra.
	
	Our intention in this paper is to prove that, under a natural geometric assumption on the group action, which in particular in satisfied in the setting considered in \cite{GWY}, this index does indeed vanish in the presence of uniformly positive scalar curvature. More precisely, we prove:	

	\begin{theorem}
		\label{thm:main}
		Let $(M^n,g)$ be a Riemannian manifold with bounded Riemannian geometry and $\Gamma$ a countable discrete group acting properly and isometrically on $M$, satisfying Assumption \ref{ass:condition}. Suppose that $M$ has a $\Gamma$-equivariant spin structure, with $S$ and $D$ the spinor bundle and Dirac operator respectively. If $M$ has uniformly positive scalar curvature, then
		$$\ind_\textnormal{max}(D)=0\in K_n(C^*_{\textnormal{max}}(M)^\Gamma),$$
		where $C^*_{\textnormal{max}}(M)^\Gamma$ is the maximal equivariant Roe algebra of $M$.
	\end{theorem}
	We remark that this result is new even when $\Gamma$ is the trivial group.
	
	Let us give a brief overview of the paper. In section \ref{sec:Roe}, we recall some standard notions from higher index theory, as well as describe the two geometric conditions alluded to above. We show that, under these geometric conditions, it is possible to define various related versions of the maximal equivariant Roe algebra of $M$, in particular the \emph{maximal equivariant uniform Roe algebra}. In section \ref{sec:regularity}, we show that this algebra can be viewed as a Hilbert module on which the operator $D$ acts. A key point is that this operator is regular and essentially self-adjoint and hence admits a functional calculus. This allows us to define, in section \ref{sec:maxuniform}, the maximal equivariant \emph{uniform} index of $D$, which we show vanishes under the positive scalar curvature assumption. This, together with the fact that the equivariant uniform Roe algebra is a subalgebra of the maximal equivariant Roe algebra, implies Theorem \ref{thm:main}.
	\hfill\vskip 0.3in
	\section{Maximal equivariant Roe algebras}
	\label{sec:Roe}
	\subsection{Notation}
	Throughout this paper, $\Gamma$ is a countable discrete group acting properly and isometrically on a Riemannian manifold $M$ of positive dimension and which is $\Gamma$-equivariantly spin. Let $S\to M$ be the spinor bundle.
	
	We will write $B(M)$, $C_b(M)$, and $C_0(M)$ respectively for the $C^*$-algebras of complex-valued functions on $M$ that are: bounded Borel, bounded continuous, and continuous and vanishing at infinity. A superscript `$^\infty$' may be added to the latter two algebras to indicate the additional requirement that its elements be smooth. 
	
	We will use $d$ to denote the Riemannian distance on $M$, and $B_r(S)$ to denote the open ball in $M$ of radius $r$ around a subset $S\subseteq M$. For any two sets $A$ and $B$, we will write pr$_1$ and pr$_2$ for the projections of the cartesian product $A\times B$, or subsets thereof, onto $A$ and $B$ respectively. For any set $S$, let $\#S$ denote its cardinality.
	
	The $\Gamma$-action on $M$ naturally determines a $\Gamma$-action on spaces of functions over $M$: for a function $f$ and $g\in\Gamma$, let $g\cdot f$ be the function given by $g\cdot f(x)=f(g^{-1}x)$. More generally, for a section $s$ of a $\Gamma$-vector bundle over $M$, define $g\cdot s(x)=g(s(g^{-1}x))$. The action of $\Gamma$ on $l^2(\Gamma)$ will always be given by the left-regular representation.
	
	\subsection{Geometric conditions}
	\hfill\vskip 0.05in
	\label{subsec:condition}
	\noindent We now make precise the two geometric conditions alluded to above. The first is that the manifold $M$ has bounded Riemannian geometry. Under this assumption, it was established in \cite{GWY} that, when the acting group $\Gamma$ is trivial, the maximal Roe algebra (see Definition \ref{def:maximalnorm}) is well-defined. However, to prove the analogous result for general $\Gamma$, we will need a second hypothesis (even if the action is free). This is given by Assumption \ref{ass:condition}.
		\subsubsection{Bounded geometry}
	\label{subsubsec:bg}
	\begin{definition} A Riemanniann manifold $M$ is said to have \emph{bounded Riemannian geometry} if it has positive injectivity radius, and the curvature tensor and each of its covariant derivatives is uniformly bounded across $M$.	
	\end{definition}
	There is also a notion of bounded geometry for discrete metric spaces:
	\begin{definition}
	\label{def:discretebddgeom}
	A discrete metric space $(X,d)$ 	is said to have bounded geometry if for any $r>0$ there exists $N_r>0$ such that for any $x\in X$,
	$$\#B_r(x)=\{x'\in X\,|\,d(x,x')\leq r\}\leq N_r,$$
	\end{definition}

	Suppose $M$ has bounded Riemannian geometry, and let $d$ be the Riemannian distance. Then the metric space $(M,d)$ contains a countable discrete subspace $X_M$ with bounded geometry such that, for some constants $c>r_0>0$, we have:
	\begin{itemize}
	\item $B_c(X_M)=M$, or that $X_M$ is $c$\emph{-dense};
	\item for all $x,y\in X_M$,
	\begin{equation}
	\label{eq:r}
	B_{r_0}(x)\cap B_{r_0}(y)\neq\emptyset\implies x=y.
	\end{equation}
	\end{itemize}
	For any $r\geq 0$, let us define the following two quantities:
	\begin{equation}
	\label{eq VR UR}
	U_r\coloneqq\inf_{x\in M}\left\{\textnormal{vol }B_{r}(x)\right\},\qquad V_r\coloneqq\sup_{x\in M}\left\{\textnormal{vol }B_{r}(x)\right\}.
	\end{equation}
	Note that bounded Riemannian geometry implies that for each $r>0$,
	$$0<U_r\leq V_r<\infty.$$
	
	The fact that the Ricci curvature of $M$ is bounded from below means that it satisfies the following volume estimate \cite[Lemma 7.1.3]{Petersen}:
	\begin{lemma}
	\label{lem:gromovbishop}
	There exist constants $C_1,C_2>0$ such that for any $r\geq 0$,
	$$V_r\leq C_1 e^{C_2 r}.$$
	\end{lemma}
	It follows from this, and the definition of $X_M$, that:

	\begin{lemma}
	\label{lem:smallintobig}
	There exists a constant $C_0$ such that for each $x\in M$ and $R>0$,
	$$\#(B_R(x)\cap X_M)\leq C_0V_{R}/U_{r_0}.$$
	\end{lemma}
	\vskip 0.05in		
\subsubsection{A condition on the $\Gamma$-action}
	\hfill\vskip 0.05in
	\noindent 
To state the second geometric condition, we use the notion of a \emph{basic domain}. Observe that, since $\Gamma$ acts properly on $M$, each $x\in M$ is contained in a $\Gamma$-invariant neighborhood of the form $\Gamma\overline{V}\cong\Gamma\times_{\Gamma_x}\overline{V}$. Here $V$ is an open $\Gamma_x$-invariant neighborhood of $x$ with compact closure, $\Gamma_x$ is the stabilizer of $x$, and $\Gamma\times_{\Gamma_x}\overline{V}$ denotes the quotient of $\Gamma\times\overline{V}$ by the relation $(g,y)\sim(gh^{-1},hy)$, for $h\in\Gamma_x$, $g\in\Gamma$, and $y\in\overline{V}$; the isomorphism is given by sending $gy$ to the class $[g,y]$. From this it follows that $M$ can be written as a disjoint union $M\cong\sqcup_i\,\Gamma N_i$, where  each $N_i$ is a Borel subset of $M$ that is preserved by the action of an isotropy subgroup $F_i$, and for each $i$ we have a $\Gamma$-equivariant homeomorphism $\Gamma\overline{N}_i\cong\Gamma\times_{F_i}\overline{N}_i$. 
	\begin{definition}
	We call $N\coloneqq\cup_i N_i\subseteq M$ the \emph{basic domain} for the decomposition $M\cong\sqcup_i\,\Gamma N_i$ given above.
	\end{definition}
    
	For the rest of this paper, we will make the following standing assumption:
	\begin{assumption}
	\label{ass:condition}
	There exists a basic domain $N$ for the $\Gamma$-action on $M$ such that
	\begin{equation*}
	l(g)\rightarrow\infty\implies d(N, gN)\rightarrow\infty,
	\end{equation*}
	where $l$ is a fixed proper length function on $\Gamma$, and $d$ is the Riemannian distance.
	\end{assumption}

 	We remark that Assumption \ref{ass:condition} is satisfied when the $\Gamma$-action on $M$ is cocompact, or when $M$ is a cocompact manifold to which an infinite cylinder is attached. In particular, it is satisfied in the situation studied in \cite{CWY}. It is also satisfied by any action of a finite group on a manifold of bounded Riemannian geometry. 
 	
 	Suppose the $\Gamma$-action on $M$ satisfies Assumption \ref{ass:condition}. Then we have the following two easy consequences.
 	\begin{corollary}
 	\label{cor:finiteisotropy}
 	$M$ decomposes into finitely many orbit types.
	\end{corollary}
 	\begin{corollary}
 	\label{cor:noaccumulation}
 	For each $x\in M$ and $R>0$, there exists $C_R>0$ such that
	$$\#(B_R(x)\cap G\cdot x)\leq C_R.$$
 	\end{corollary}
	\begin{proof}
	Observe that Assumption \ref{ass:condition} implies that this relation holds for $x\in N$. The general statement follows by observing that for any $g\in\Gamma$, 
	\begin{align*}
	\#(B_R(x)\cap G\cdot x)&=\#(B_R(g\cdot x)\cap G\cdot x).\qedhere
	\end{align*}
	\end{proof}

	\subsection{Operator algebras}
 	\hfill\vskip 0.05in
	\noindent We now recall the definitions of geometric modules and Roe algebras, with the goal of proving that, under the previously stated geometric assumptions, the maximal equivariant Roe algebra is well-defined. We also provide an estimate of the maximal norm that will be important in section \ref{sec:regularity} (see Lemma \ref{lem:keyestimate}).

	\begin{definition}
	An $M$-$\Gamma$-module is a separable Hilbert space $\mathcal{H}$ equipped with a non-degenerate $*$-representation $\rho\colon C_0(M)\rightarrow\mathcal{B}(\mathcal{H})$ and a unitary representation $U\colon\Gamma\rightarrow\mathcal{U}(\mathcal{H})$ such that for all $f\in C_0(M)$ and $g\in\Gamma$, $U_g\rho(f)U_g^*=\rho(g\cdot f)$.
	\end{definition}
	For brevity, we will omit $\rho$ from the notation when it is clear from context.
	\begin{definition}
	\label{def:suppprop}
		Let $\mathcal{H}$ be an $M$-$\Gamma$-module and $T\in\mathcal{B}(\mathcal{H})$.
		\begin{itemize}
			\item The \emph{support} of $T$, denoted $\textnormal{supp}(T)$, is the complement in $M\times M$ of the set of $(x,y)$ for which there exist $f_1,f_2\in C_0(M)$ with $f_1(x)\neq 0$, $f_2(y)\neq 0$, and
			$$f_1Tf_2=0;$$
			\item The \textit{propagation} of $T$ is the extended real number $$\textnormal{prop}(T)=\sup\{d(x,y)\,|\,(x,y)\in\textnormal{supp}(T)\};$$
			\item $T$ is \textit{locally compact} if $fT$ and $Tf\in\mathcal{K}(H)$ for all $f\in C_0(M)$;
			\item $T$ is \textit{$\Gamma$-invariant} if $U_g TU_g^*=T$ for all $g\in\Gamma$.
		\end{itemize}
		Let $\mathbb{C}[M;\mathcal{H}]^\Gamma\subseteq\mathcal{B}(\mathcal{H})$ be the $*$-subalgebra of $\Gamma$-invariant, locally compact operators with finite propagation.
	\end{definition}

	We will work with certain maximal completions of $\mathbb{C}[M;\mathcal{H}]^{\Gamma}$. In order show that such completions are well-defined, we restrict ourselves to those modules $\mathcal{H}$ that satisfy an additional admissibility condition. To state this, we need the following fact: if $H$ is a Hilbert space and $\rho\colon C_0(M)\rightarrow\mathcal{B}(H)$ a non-degenerate $*$-representation, then $\rho$ extends uniquely to a $*$-representation $\tilde{\rho}\colon B(M)\rightarrow\mathcal{B}(H)$ subject to the property that, for a uniformly bounded sequence in $B(M)$ converging pointwise, the corresponding sequence in $\mathcal{B}(H)$ converges in the strong topology.

	\begin{definition}[\cite{Yu}]
		An $M$-$\Gamma$-module $\mathcal{H}$ is \emph{admissible} if:
		\begin{enumerate}[(i)]
		\item For any non-zero $f\in C_0(M)$, $\pi(f)\notin\mathcal{K}(\mathcal{H})$;
		\item For any finite subgroup $F$ of $\Gamma$ and any $F$-invariant Borel subset $E$ of $M$, there is a Hilbert space $\mathcal{H}'$ equipped with the trivial $F$-representation such that $\tilde{\pi}(\mathbbm{1}_E)\mathcal{H}'\cong l^2(F)\otimes\mathcal{H}'$ as $F$-representations, where $\tilde{\pi}$ is defined by extending $\pi$ as above.
		\end{enumerate}
	\end{definition}
	If an $M$-$\Gamma$-module $\mathcal{H}$ is admissible, we will write $\mathbb{C}[M]^\Gamma\coloneqq\mathbb{C}[M;\mathcal{H}]^\Gamma$, noting that $\mathbb{C}[M]^\Gamma$ is independent of the choice of admissible module.

	In this paper we will use two $M$-$\Gamma$-modules. The first is $L^2(S)$, equipped with the natural $\Gamma$-action and $C_0(M)$-representation. In general, $L^2(S)$ is not admissible (see also Remark \ref{rem:freeL2}). The second is the space
		$$\mathcal{H}_M\coloneqq L^2(S)\otimes l^2(\Gamma),$$
		equipped with the multiplicative action of $C_0(M)$ on the first factor and the diagonal $\Gamma$-action. Since we are assuming $M$ to have positive dimension, one verifies that $L^2(S)\otimes l^2(\Gamma)$ is an admissible $M$-$\Gamma$-module.
	
	We can view $L^2(S)$ as a submodule of $\mathcal{H}_M$ in the following way. Let $\chi\in C^\infty(M)$ be a cut-off function, meaning that $\textnormal{supp}(\chi)$ has compact intersection with every $\Gamma$-orbit and that for all $x\in M$, we have $\sum_{g\in\Gamma}\chi(gx)^2=1.$ Note that this sum is finite by properness of the action. Then the map 
	$$j\colon L^2(S)\rightarrow \mathcal{H}_M,\qquad j(s)(x,g)=\chi(g^{-1}x)s(x)$$
	is a $\Gamma$-equivariant isometric embedding. Let $p\colon\mathcal{H}_M\rightarrow j(L^2(S))$ be the orthogonal projection associated to $j$. 
    On operators, $j$ induces a map taking $T\mapsto jTj^{-1}p$, and we will denote this by
    \begin{equation}
    \label{eq oplus 0}
	\oplus\,0\colon\mathcal{B}(L^2(S))\rightarrow\mathcal{B}(\mathcal{H}_M).
    \end{equation}
    It is an injective $*$-homomorphism that preserves $\Gamma$-equivariance, local compactness, as well as finiteness of propagation, and hence restricts to an injective $*$-homomorphism 
    $\mathbb{C}[M,L^2(S)]^\Gamma\hookrightarrow\mathbb{C}[M]^\Gamma$.

	\begin{remark}
	\label{rem:freeL2}
		When the $\Gamma$-action $M$ is both proper and free, $L^2(S)$ is itself an admissible $M$-$\Gamma$-module, and the discussion above simplifies.
	\end{remark}
	\begin{definition}
	\label{def:maximalnorm}
	For an operator $T\in\mathbb{C}[M]^\Gamma$, its \emph{maximal norm} is
	$$||T||_{\textnormal{max}}\coloneqq\sup_{\phi,H'}\left\{\norm{\phi(T)}_{\mathcal{B}(H')}\,|\,\phi\colon\mathbb{C}[M]^\Gamma\rightarrow\mathcal{B}(H')\textnormal{ is a $*$-representation}\right\}.$$
	The \emph{maximal equivariant Roe algebra of $M$}, denoted $C^*_{\text{max}}(M)^\Gamma$, is the completion of $\mathbb{C}[M]^\Gamma$ with respect to $||\cdot||_{\textnormal{max}}$.
	\end{definition} 

 	\subsection{Estimating the maximal norm}
 	\label{subsec:estimating}
 	\hfill\vskip 0.05in
 	\noindent To make sense of Definition \ref{def:maximalnorm}, one needs to show that for any $T\in\mathbb{C}[M]^\Gamma$, there exists a constant $C$ bounding the norm of $T$ in any $*$-representation. We now show that this is the case under the geometric conditions in subsection \ref{subsec:condition}, namely:
	\begin{proposition}
	\label{prop:finiteness}
	Suppose that $M$ has bounded geometry and that Assumption \ref{ass:condition} holds. Then for any $T\in\mathbb{C}[M]^\Gamma$ and any $*$-homomorphism $\phi\colon\mathbb{C}[M]^\Gamma\rightarrow\mathcal{B}(H')$, for $H'$ a Hilbert space, we have
		$$\norm{\phi(T)}_{\mathcal{B}(H')}<\infty.$$	
	\end{proposition}
	In order to perform the estimates required to prove this, we will work with the module $\mathcal{H}_M$ in a discretized form that we now describe.
	\subsubsection{Discretizing $\mathcal{H}_M$}
	\label{subsubsec:discretizing}
	\hfill\vskip 0.05in	
	\noindent 
	
	Let us consider the case when $S$ is the trivial line bundle over $M$, with the general case being analogous. Let $X_M$, $c$, and $r_0$ be as in \ref{subsubsec:bg}. The fact that $X_M$ is $c$-dense, together with \eqref{eq:r}, implies that there exists a Borel cover $\mathcal{U}$ of $M$ such that, for each $U\in\mathcal{U}$, there exists $x\in X_M$ with $B_{r_0}(x)\subseteq U\subseteq B_c(x)$. 
	
	Let $\pi\colon M\rightarrow M/\Gamma$ denote the projection onto orbits. Using the cover $\mathcal{U}$, we can construct a subset $X_0$ of $X_M$ with the following properties:
	\begin{itemize}
	\item $B_c(\Gamma\cdot X_0)=M$; 
	\item There exists a constant $C>0$ such that for any $x\in X_0$ and $R>0$,
		$$\#\pi(B_R(\Gamma\cdot x)\cap(\Gamma\cdot X_0))\leq CV_{R}/U_{r_0},$$
	\end{itemize}
	where $V_{R}$ and $U_{r_0}$ are as in \eqref{eq VR UR}.
Define
	\begin{equation}
	\label{eq Z}	
	Z\coloneqq\Gamma\cdot X_0.
	\end{equation}
	
We can use the set $Z$ to rewrite the module $\mathcal{H}_M\cong L^2(M\times\Gamma)$ as follows. Since the (diagonal) action of $\Gamma$ on $M\times\Gamma$ is proper and free, it admits a fundamental domain $D_0\subseteq M\times\Gamma$. We may choose $D_0$ so that pr$_1(D_0)\subseteq N$, where $N$ is the basic domain in Assumption \ref{ass:condition}. Then the set $D_1\coloneqq D_0\cap(Z\times\Gamma)$ is a fundamental domain for the $\Gamma$-subspace $Z\times\Gamma\subseteq M\times\Gamma$. We may choose $D_0$ in such a way that it contains $D_1$ as a $c$-dense subset. This defines for us a a unitary isomorphism 
$$L^2(D_0)\cong l^2(D_1)\otimes H$$ 
for a separable Hilbert space $H$. In turn we have $\Gamma$-equivariant unitary isomorphisms
        \begin{align}
        \label{eq:discretization}
\mathcal{H}_M&\cong L^2(M\times\Gamma)\nonumber\\
&\cong L^2(D_0)\otimes l^2(\Gamma)\nonumber\\
&\cong l^2(D_1)\otimes l^2(\Gamma)\otimes H\nonumber\\
&\cong l^2(Z\times\Gamma)\otimes H,
        \end{align}
        where $H$ is equipped with the trivial $\Gamma$-representation.
Now given a point $y\in Z$, let $O_y\subseteq Z$ and $F_y<\Gamma$ denote its orbit and stabilizer respectively. For each such $y$, identify (set-theoretically) $O_y$ with $\Gamma/F_y$ to obtain a bijective map
	$$\varphi_y\colon O_y\times F_y\xrightarrow{\cong}\Gamma/F_y\times F_y.$$
	Choose a section $\phi_y\colon\Gamma/F_y\rightarrow\Gamma$. Then we have a bijection 
	\begin{align*}
	\tilde{\phi}_y\colon\Gamma &\rightarrow\Gamma/F_y\times F_y,\\
	g&\mapsto(gF_y,g^{-1}\phi_y(gF_y)).
	\end{align*}
	Let $\Gamma$ act on $\Gamma/F_y\times F_y$ by the pushforward of the $\Gamma$-action on itself along $\tilde{\phi}_y$. Now consider the collection of orbits $W\coloneqq\pi(Z)\subseteq M/\Gamma$. For each $O\in W$, choose a representative in the basic domain $N$, and let $Y$ be the collection of representatives so obtained as $O$ ranges over $W$. Define the sets
	\begin{equation}
	\label{eq tilde Z}
	\tilde{Z}\coloneqq\bigsqcup_{y\in Y}(O_y\times F_y),\qquad E\coloneqq\bigsqcup_{y\in Y}(\Gamma/F_y\times F_y),
	\end{equation}
	and equip them with piecewise $\Gamma$-actions.
	Upon taking a disjoint union of the maps $\varphi_y$, we obtain a $\Gamma$-equivariant bijection
	$$\varphi\colon\tilde{Z}\xrightarrow{\cong}E.$$
	This in turn gives equivalent $\Gamma$-representations on the Hilbert spaces
	$$\bigoplus_{y\in Y}l^2(O_y\times F_y)\cong\bigoplus_{y\in Y}l^2(\Gamma/F_y\times F_y).$$
	We have the following:
\begin{proposition}
\label{prop:Hilbertiso}
There is a $\Gamma$-equivariant unitary isomorphism
	$$\mathcal{H}_M\cong l^2(\tilde{Z})\otimes H,$$
	where $H$ is a Hilbert space equipped with the trivial $\Gamma$-representation.
\end{proposition}
\begin{proof}
By \eqref{eq:discretization}, we have 
\begin{equation}
\label{eq:intermediate}
\mathcal{H}_M\cong\bigoplus_{y\in Y}l^2(O_y\times\Gamma)\otimes H\cong\bigoplus_{y\in Y}l^2(\Gamma/F_y\times\Gamma)\otimes H.
\end{equation}

For each $y\in Y$, let $\nu_y\colon F_y\backslash\Gamma\rightarrow\Gamma$ be a section. Let $\phi_y$ and $\tilde{\phi}_y$ be given as above. One verifies that the following map is a $\Gamma$-equivariant bijection:
	$$\vartheta_y\colon\Gamma/F_y\times\Gamma\to\Gamma/F_y\times F_y\times F_y\backslash\Gamma,$$
	$$(g_1F_y,g_2)\mapsto(g_1 F_y,\nu_y(F_y g_1^{-1}g_2)g_2^{-1}\phi_y(g_2\nu_y(F_y g_1^{-1}g_2)^{-1}F_y),F_y g_1^{-1}g_2),$$
	where $\Gamma$ acts on the left diagonally, while on the right it acts on $\Gamma/F_y\times F_y$ by pushing forward the left-action of $\Gamma$ on itself along $\tilde{\phi}_y$, and trivially on $F_y\backslash\Gamma$. Indeed, $\vartheta_y$ can be written as a composition of maps as follows. Define the bijection		
	$$\tilde{\nu}_y\colon\Gamma\to F_y\times F_y\backslash\Gamma,\qquad g\mapsto(g\nu_y(F_y g)^{-1},F_y g).$$
	Then $\vartheta_y$ is the following composition of bijections:
    $$\Gamma/F_y\times\Gamma\rightarrow\Gamma\times_{F_y}\Gamma\rightarrow(\Gamma\times_{F_y}F_y)\times F_y\backslash\Gamma\rightarrow\Gamma\times F_y\backslash\Gamma\rightarrow\Gamma/F_y\times F_y\times F_y\backslash\Gamma,$$
    \begin{align*}
    (g_1 F_y,g_2)&\mapsto[(g_1,g_1^{-1}g_2)]\\
    &\mapsto([g_1,\textnormal{pr}_1(\tilde{\nu}_y(g_1^{-1}g_2))],\textnormal{pr}_2(\tilde{\phi}_y(g_1^{-1}g_2)))\\
    &\mapsto(g_2\nu_y(F_y g_1^{-1}g_2)^{-1},F_y g_1^{-1}g_2)\\
    &\mapsto(\tilde{\phi}_y(g_2\nu_y(F_y g_1^{-1}g_2)^{-1}),F_y g_1^{-1}g_2),
    \end{align*}
	where $\Gamma\times_{F_y}\Gamma$ and $\Gamma\times_{F_y}F_y$ are, respectively, the quotients of $\Gamma\times\Gamma$ and $\Gamma\times F_y$ by the equivalence relation $(g_1,g_2)\sim(g_1 k^{-1},kg_2),$ for $g_1\in\Gamma$, $k\in F_y$, and $g_2$ in either $\Gamma$ or $F_y$. For each $y\in Y$, $\vartheta_y$ induces a $\Gamma$-equivariant isomorphism
    $$l^2(\Gamma/F_y\times\Gamma)\cong l^2(\Gamma/F_y\times F_y)\otimes l^2(F_y\backslash\Gamma),$$
    where $\Gamma$ acts trivially on $l^2(F_y\backslash\Gamma)$. Combining this with \eqref{eq:intermediate} gives
	\begin{align*}
	\mathcal{H}_M&\cong\bigoplus_{y\in Y}l^2(\Gamma/F_y\times F_y)\otimes l^2(F_y\backslash\Gamma)\otimes H\\
	&\cong\bigoplus_{y\in Y}l^2(O_y\times F_y)\otimes l^2(F_y\backslash\Gamma)\otimes H.
	\end{align*}
Now, for each $y$, pick an identification of $l^2(F_y\backslash\Gamma)\otimes H$ with $H$ to give
	\begin{align*}
	\mathcal{H}_M&\cong\bigoplus_{y\in Y}l^2(O_y\times F_y)\otimes H\cong l^2(\tilde{Z})\otimes H.\qedhere
	\end{align*}
\end{proof}
The isomorphism $\mathcal{H}_M\cong l^2(\tilde{Z})\otimes H$ constructed in the above proof gives a $\Gamma$-equivariant identification between operators in $\mathcal{B}(\mathcal{H}_M)$ and $\tilde{Z}\times\tilde{Z}$-matrices with entries in $\mathcal{B}(H)$. Furthermore, it imposes a strong relationship between the propagation of an operator in $\mathcal{B}(\mathcal{H}_M)$ and the off-diagonal support of the corresponding matrix. To make this precise, we introduce the following notion.
\begin{definition}
The \emph{matricial support} of $T\in\mathcal{B}(\mathcal{H}_M)$ is the set
	$$\textnormal{matsupp}(T)=\{(w,z)\in\tilde{Z}\times\tilde{Z}\,|\,T_{wz}\neq 0\}.$$
\end{definition}
%
Define the composition
$$\textnormal{pr}\colon\tilde{Z}\xrightarrow{\textnormal{pr}_1}\bigsqcup_{y\in Y}O_y\hookrightarrow M.$$
For subsets $S,S'$ of $\tilde{Z}$, let us write
\begin{align*}
d^{\tilde{Z}}(S,S')&\coloneqq d(\textnormal{pr}(S),\textnormal{pr}(S')).
\end{align*}
(If $S$ or $S'=\{z\}$, we will write $z$ in place of $S$ or $S'$.) Using the fact that $Z$ is $c$-dense in $M$, we see that for any $w,z\in\textnormal{matsupp}(T)$,
	$$d^{\tilde{Z}}(w,z)\leq\textnormal{prop}(T)+c.$$

Finally, observe that the subset
	$$\mathcal{F}\coloneqq\left\{(x,e)\in\tilde{Z}\,|\,x\in Y\right\}$$ 
	is a fundamental domain for the $\Gamma$-action on $\tilde{Z}$, where $e$ is the identity in $\Gamma$. Thus if $T\in\mathcal{B}(\mathcal{H}_M)$ is a $\Gamma$-invariant operator, it is determined entirely by its entries in $\tilde{Z}\times\mathcal{F}$. If, in addition, $T$ has finite propagation, then one only needs to know the entries in the subset $B^{\tilde{Z}}_{\textnormal{prop}(T)+c}(\mathcal{F})\times\mathcal{F},$
		where 
		$$B^{\tilde{Z}}_R(S)\coloneqq\{z\in\tilde{Z}\,|\,d^{\tilde{Z}}(z,S)<R\},$$
		for a subset $S\subseteq\tilde{Z}$ and $R>0$.


\subsubsection{Norm estimation} 
	\hfill\vskip 0.05in	
\noindent We now proceed with the proof of Proposition \ref{prop:finiteness}. The first observation is:
\label{subsubsec:a}
	\begin{lemma}
	\label{lem:interactions}
	There exists a constant $C$ such that for any $z\in\mathcal{F}$ and $R>0$, 
	$$\#B_R^{\tilde{Z}}(z)\leq CV^2_{R}.$$
	\end{lemma}
	\begin{proof} 
	By Corollary \ref{cor:finiteisotropy}, the cardinality of the stabilizer $F_x$ of any point $x\in M$ is uniformly bounded. Thus it suffices to show that there exists $C$ such that	$$\#\textnormal{pr}(B_R^{\tilde{Z}}(z))\leq CV^2_{R}$$
	for any $z\in\mathcal{F}$ and $R>0$. In other words, it suffices to show that for any $x\in N$, $\#(B_R(x)\cap Z)\leq CV_R^2$, where $Z$ is as in \eqref{eq Z}. To this end, let $\mathcal{U}$ be the Borel cover from \ref{subsubsec:discretizing}. Lemma \ref{lem:smallintobig} implies that there exists $C_1$ such that
	$$\#\{U\in\mathcal{U}\colon U\cap B_{R}(x)\neq\emptyset\}\leq C_1V_{R}/U_{r_0}.$$ 
	Corollary \ref{cor:noaccumulation} implies that each element of $\mathcal{U}$ contains at most $C_{2c}$ points from any single $\Gamma$-orbit $O\subseteq M$ (here $C_{2c}$ is the constant $C_R$ from Corollary \ref{cor:noaccumulation} with $R=2c$). Thus
		$$\#(B_R(x)\cap O)\leq V_{R+c}C_{2c}/U_{r_0}.$$
	By construction, the number of orbits in the set $Z$ that intersect $B_{R}(x)$ is bounded above by $C_2V_{R}/U_{r_0}$ for some constant $C_2$. It follows from Lemma \ref{lem:gromovbishop} that 
	$$\#(B_R(x)\cap Z)\leq C V_{R}^2,$$ 
	where $C$ is independent of $R$.
	\end{proof}
	\begin{lemma}
	\label{lem:interactionsfundamentaldomain}
	There exists a constant $C$ such that for any $z\in\mathcal{F}$ and $R>0$, 
	$$\#(B_R^{\tilde{Z}}(z)\cap\mathcal{F})\leq CV_{R},$$
	where $\tilde Z$ is as in \eqref{eq tilde Z}.
	\end{lemma}
	\begin{proof}
	This follows from the proof of Lemma \ref{lem:interactions}, but without the need to consider the orbit direction.
	\end{proof}
	We are led to the following lemma, which in particular implies Proposition \ref{prop:finiteness}:
	\begin{lemma}
	\label{lem:keyestimate}
	For any $T\in\mathbb{C}[M]^\Gamma$ and any $*$-homomorphism $\phi\colon\mathbb{C}[M]^\Gamma\rightarrow\mathcal{B}(H')$, where $H'$ is a Hilbert space,
		$$\norm{\phi(T)}_{\mathcal{B}(H')}\leq CC_TV^4_{\textnormal{prop}(T)},$$
		where $C_T\coloneqq\sup_{w,z\in\tilde{Z}}\norm{T_{wz}}$ and $C$ is a constant independent of $T$.
	\end{lemma}
	\begin{proof}
	Let $\mathcal{T}$ denote the operator whose $\tilde{Z}\times\tilde{Z}$-matrix entries are equal to those of $T$ on $B^{\tilde{Z}}_{\textnormal{prop}(T)+c}(\mathcal{F})\times\mathcal{F}$, with all others being zero. It follows from the proof of \cite[Lemma 3.4]{GWY} that we can write $\mathcal{F}$ as a disjoint union of subsets $\mathcal{F}_1,\mathcal{F}_2,\ldots,\mathcal{F}_{L_1+1}$ with the property that if $w,z\in\mathcal{F}_i$ for some $i$, then $d^{\tilde{Z}}(w,z)>2\,\textnormal{prop}(T)+3c$. Let 
	$$\mathcal{Q}=\{(z',z)\in\tilde{Z}\times\mathcal{F}\,|\,d^{\tilde{Z}}(z',z)\leq \,\textnormal{prop}(T)+c\}.$$
	Write $\mathcal{Q}_i=\mathcal{Q}\cap(\tilde{Z}\times\mathcal{F}_i)$. By Lemmas \ref{lem:interactions}, \ref{lem:interactionsfundamentaldomain}, and \ref{lem:gromovbishop}, there exist $C_1,C_2>0$ such that for any $z\in\mathcal{F}$,
	\begin{align*}
	\#(B^{\tilde{Z}}_{2\,\textnormal{prop}(T)+3c}(z)\cap\mathcal{F})&\leq C_1 V_{2\,\textnormal{prop}(T)},\\
	\#B^{\tilde{Z}}_{\,\textnormal{prop}(T)+c}(z)&\leq C_2 V^2_{\,\textnormal{prop}(T)}.
	\end{align*}
	Note that $C_1$ and $C_2$ are independent of $T$. Setting 
	$$L_1\coloneqq\big\lfloor C_1 V_{2\,\textnormal{prop}(T)}\big\rfloor,\qquad L_2\coloneqq\big\lfloor C_2 V^2_{\,\textnormal{prop}(T)}\big\rfloor,$$ 
	one sees that for any $z\in\mathcal{F}_i$, there are at most $L_2$ elements $z'\in\tilde{Z}$ such that $(z',z)\in\mathcal{Q}_i$. Thus there exists a disjoint union decomposition
	$$\mathcal{Q}=\bigsqcup_{i=1}^{(L_1+1)L_2}\mathcal{P}_i,$$
	where the sets $\mathcal{P}_i$ have the property that, for any two distinct elements $(w,z)$ and $(w',z')\in\mathcal{P}_i$, we have $d^{\tilde{Z}}(z,z')>2\,\textnormal{prop}(T)+3c$, and hence $w\neq w'$. This gives a decomposition $$\mathcal{T}=\sum_{i=1}^{(L_1+1)L_2}\mathcal{T}_i,\quad \textnormal{matsupp}(\mathcal{T}_i)\subseteq\mathcal{P}_i.$$
	Observe that, for each $i$, the operator $\mathcal{T}_i^*\mathcal{T}_i$ is $\Gamma$-invariant and has matricial support confined to the diagonal of $\mathcal{F}\times\mathcal{F}$ and so has norm at most $C_{T}^2.$
	 
	Let $l^\infty(\tilde{Z},\mathcal{K}(H))^\Gamma\subseteq\mathbb{C}[M]^\Gamma$ denote the $*$-subalgebra of $\Gamma$-invariant operators whose matrix entries belong to the diagonal of $\tilde{Z}\times\tilde{Z}$. Since it is a $C^*$-algebra, the norm of any operator $T'\in l^\infty(\tilde{Z},\mathcal{K}(H))^\Gamma$ contracts under any $*$-representation of $\mathbb{C}[M]^\Gamma$. Applying this to $T'=\mathcal{T}_i^*\mathcal{T}_i$, and using Lemma \ref{lem:gromovbishop}, we get:
	\begin{align*}
	\norm{\phi(T)}_{\mathcal{B}(H')}&\leq\sum_{i=1}^{(L_1+1)L_2}\norm{\phi(\mathcal{T}_i)}_{\mathcal{B}(H')}
\leq (L_1+1)L_2 C_T\leq CC_T V^4_{\textnormal{prop}(T)},
	\end{align*}
	for some $C$ independent of $T$.
	\end{proof}
		\subsection{Maximal Roe algebras}
		\label{subsec:maxRoe}
	\hfill\vskip 0.05in	
\noindent It follows from Lemma \ref{lem:keyestimate} that the norm of an operator $T\in\mathbb{C}[M]^\Gamma$ in any $*$-representation has a finite bound independent of the $*$-representation. This allow us to define several versions of the maximal equivariant Roe algebra, as follows.

The first is the algebra $C^*_{\textnormal{max}}(M)^\Gamma$ defined on the admissible module $\mathcal{H}_M$ as the completion of $\mathbb{C}[M]^\Gamma$ in the norm $\norm{\cdot}_{\textnormal{max}}$ (see Definition \ref{def:maximalnorm}). 

We also have:
\begin{definition}
The maximal equivariant Roe algebra on the $M$-$\Gamma$-module $L^2(S)$, denoted by $C^*_{\textnormal{max}}(M;L^2(S))^\Gamma$, is the completion of $\mathbb{C}[M;L^2(S)]^\Gamma$ under the norm pulled back under the injective $*$-homomorphism given by the composition
$$\mathbb{C}[M;L^2(S)]^\Gamma\xrightarrow{\oplus\,0}\mathbb{C}[M;L^2(S)]\oplus 0\hookrightarrow\mathbb{C}^*_{\textnormal{max}}(M)^\Gamma.$$
	\end{definition}

Finally, we define a \emph{uniform} version of the maximal Roe algebra on the module $L^2(S)$ as the completion in $C^*_{\textnormal{max}}(M;L^2(S))^\Gamma$ of a certain space Schwartz kernels. This algebra will play a key role in the results of the next section. Recall:

	\begin{definition}
		A section $k$ of $\End(S)\cong S\boxtimes S^*$ has \emph{finite propagation} if there exists an $R>0$ such that for all $x, y\in M$,
		$$d(x, y)>R\implies k(x, y)=0.$$ 
		The infimum of such $R$ is called the \emph{propagation} of $k$, denoted by $\textnormal{prop}(k)$.
	\end{definition}
	\begin{definition}
		Let $\mathcal{S}_u^\Gamma$ denote the $*$-subalgebra of $\mathcal{B}(L^2(S))$ whose elements are given by Schwartz kernels $k\in C_b^\infty(S\boxtimes S^*)$ that satisfy:
		\begin{enumerate}[(i)]
			\item $k$ has finite propagation;
			\item $k(x,y)=k(gx,gy)$ for all $g\in\Gamma$;
			\item Each covariant derivative of $k(x,y)$ is uniformly bounded over $M\times M$.
		\end{enumerate}
	\end{definition}
	Note that properties (i) and (ii) imply that $\mathcal{S}_u^\Gamma$ is a $*$-subalgebra of $\mathbb{C}[M;L^2(S)]^\Gamma$.
		\begin{definition}
	The \emph{maximal equivariant uniform Roe algebra} of $M$ on $L^2(S)$, denoted by $C^*_{\text{max},u}(M;L^2(S))^\Gamma$, is the completion of $\mathcal{S}_u^\Gamma$ in $C^*_{\textnormal{max}}(M;L^2(S))^\Gamma$.
	\end{definition}
		
	\begin{remark}
		Elements of $\mathcal{S}_u^\Gamma$ are approximable on each local piece of the manifold by finite-rank operators in a way that is uniform across the manifold. The completion of $\mathcal{S}_u^\Gamma$ in the operator norm on $\mathcal{B}(L^2(S))$ is referred to as the \emph{reduced} equivariant uniform Roe algebra on the module $L^2(S)$.
	\end{remark}
	\hfill\vskip 0.3in

	\section{Functional calculus}
	\label{sec:regularity}
	We now use the estimates established in the previous section to complete a key step in the proof of Theorem \ref{thm:main}, namely to establish a functional calculus for the unbounded operator $D$ on the maximal equivariant uniform Roe algebra. The main result of this section is Theorem \ref{thm:regularity}. A basic reference for the material on Hilbert $C^*$-modules used in this section is \cite{Lance}.

\subsection{A Hilbert module operator}
\label{subsection hilbert operator}
	\hfill\vskip 0.05in
	\noindent 	We view the $C^*$-algebra $C^*_{\text{max},u}(M;L^2(S))^\Gamma$ as a right Hilbert module over itself. The inner product and right action on $C^*_{\text{max},u}(M;L^2(S))^\Gamma$ are defined naturally through multiplication: for $a,b\in C^*_{\text{max},u}(M;L^2(S))^\Gamma$,
	$$\langle a,b\rangle=a^*b,\qquad a\cdot b=ab,$$
	where the adjoint is defined on the kernel algebra $\mathcal{S}_u^\Gamma$ in the usual way. The algebra of compact operators on this Hilbert module can be identified with $C^*_{\textnormal{max},u}(M;L^2(S))^\Gamma$ via left multiplication. Similarly, the algebra of bounded adjointable operators can be identified with the multiplier algebra $\mathcal{M}$ of $C^*_{\textnormal{max},u}(M;L^2(S))^\Gamma$.

	We first show that $D$ can be viewed as an unbounded operator on this Hilbert module. 
	The Dirac operator $D$ acts naturally on smooth sections of $M\times M$ as follows: for each $s\in\mathcal{S}_u^\Gamma$, define $Ds$ to be the section
	$$(x,y)\mapsto D_x s(x,y),$$
	where $D_x$ means the operator $D$ acting on the $x$-coordinate. One verifies easily that $D$ is symmetric with respect to the inner product structure defined above. 
	
	In keeping with the usual notion of an unbounded operator on a Hilbert module, we need to ensure that the domain of $D$ is a right $C^*_{\textnormal{max},u}(M;L^2(S))^\Gamma$-module. To do this, let $(C^*_{\textnormal{max},u}(M;L^2(S))^\Gamma)^+$ be the unitization of $C^*_{\textnormal{max},u}(M;L^2(S))^\Gamma$. Then the right ideal $\mathcal{S}_u^\Gamma\cdot (C^*_{\textnormal{max},u}(M;L^2(S))^\Gamma)^+$ contains $\mathcal{S}_u^\Gamma$ and admits a right action by $C^*_{\textnormal{max},u}(M;L^2(S))^\Gamma$. We can extend the action of $D$ in a natural way to $\mathcal{S}_u^\Gamma\cdot(C^*_{\textnormal{max},u}(M;L^2(S))^\Gamma)^+$ by setting, for each $a\in\mathcal{S}_u^\Gamma$ and $b\in C^*_{\textnormal{max},u}(M;L^2(S))^\Gamma$,
	$$D(ab)\coloneqq (Da)b.$$
	Note that this is well-defined, since if $ab=\widetilde{a}\widetilde{b}$, then for any $v\in\mathcal{S}_u^\Gamma$, symmetricity of $D$, together with continuity of the inner product, shows that
	\begin{align*}
	\langle (Da)b,v\rangle &=\lim_{n\to\infty}\langle (Da)b_n,v\rangle\\
	&=\langle ab,Dv\rangle\\
	&=\langle\widetilde a\widetilde b,Dv\rangle\\
	&=\langle (D\widetilde a)\widetilde b,v\rangle,
	\end{align*}
	where $b_n$ is a sequence in $\mathcal{S}_u^\Gamma$ converging to $b$. Density of $\mathcal{S}_u^\Gamma$ in $C^*_{\textnormal{max},u}(M;L^2(S))^\Gamma$ then implies that $(Da)b=(D\widetilde a)\widetilde b$.
	
	After taking the closure, we obtain a densely defined, closed $C^*_{\textnormal{max},u}(M;L^2(S))^\Gamma$-linear operator
	\begin{equation}
	\label{eq:Dl}
	\overline{D}\colon C^*_{\textnormal{max},u}(M;L^2(S))^\Gamma\rightarrow C^*_{\textnormal{max},u}(M;L^2(S))^\Gamma.
	\end{equation}
	Further, for each $l\in\mathbb{N}$, the operator $\overline{D}^l$ is a densely defined, closed $C^*_{\textnormal{max},u}(M;L^2(S))^\Gamma$-linear operator on $C^*_{\textnormal{max},u}(M;L^2(S))^\Gamma$.
	
	We make two remarks. First, since the action of $D$ on $\mathcal{S}_u^\Gamma\cdot (C^*_{\textnormal{max},u}(M;L^2(S))^\Gamma)^+$ is determined by its action on $\mathcal{S}_u^\Gamma$, in practice we may just work with the latter. Second, for the sake of brevity, we will simply write $D$ to mean its closure $\overline{D}$ where confusion is unlikely to arise.
	
	\subsection{Regularity and essential self-adjointness}
	\hfill\vskip 0.05in
	\noindent Let us state main result of this section:
	\begin{theorem}
	\label{thm:regularity}
	There exists a real number $\mu\neq 0$ such that the unbounded operators
	$$D\pm\mu i\colon C^*_{\textnormal{max},u}(M;L^2(S))^\Gamma\rightarrow C^*_{\textnormal{max},u}(M;L^2(S))^\Gamma$$
	have dense range.
	\end{theorem}
	We will proceed to the proof of this theorem shortly. First observe the following consequence (\cite[Lemmas 9.7 and 9.8]{Lance}):
	\begin{corollary}\label{cor:selfadj}
	For each $l\in\mathbb{N}$, the unbounded operator $\overline{D}^l$ on the Hilbert module $C^*_{\textnormal{max},u}(M;L^2(S))^\Gamma$ is regular and self-adjoint.
	\end{corollary}
	Regular and essentially self-adjoint operators on a Hilbert $C^*$-module admit a functional calculus that satisfy the following set of properties \cite[Theorem 10.9]{Lance}, \cite[Proposition 16]{Kucerovsky}:
\begin{theorem}\label{thm:functionalcalculus}
Let $B$ be a $C^*$-algebra and $\mathcal{N}$ a Hilbert $B$-module. Let $C(\mathbb{R})$ be the $*$-algebra of complex-valued continuous functions on $\mathbb{R}$. For any regular, essentially self-adjoint operator $T$ on $\mathcal{N}$, there is a $*$-preserving linear map
	\begin{align*}
	\pi_T\colon C(\mathbb{R})&\rightarrow\mathcal{R}_B(\mathcal{N}),\\
	f&\mapsto f(T)\coloneqq\pi_T(f),
	\end{align*}
where $\mathcal{R}_B(\mathcal{N})$ denotes the regular operators on $\mathcal{N}$, such that:	\begin{enumerate}
		\item[(i)] $\pi_T$ restricts to a $*$-homomorphism $C_b(\mathbb{R})\rightarrow\mathcal{L}_B(\mathcal{N})$;
		\item[(ii)] If $|f(t)|\leq|g(t)|$ for all $t\in\mathbb{R}$, then $\textnormal{dom}(g(T))\subseteq\textnormal{dom}(f(T))$;
		\item[(iii)] If $(f_n)_{n\in\mathbb{N}}$ is a sequence in $C(\mathbb{R})$ for which there exists $F\in C(\mathbb{R})$ such that $|f_n(t)|\leq F(t)|$ for all $t\in\mathbb{R}$, and if $f_n$ converge to a limit function $f \in C(\R)$ uniformly on compact subsets of $\mathbb{R}$, then $f_n(T)x\mapsto f(T)x$ for each $x\in\textnormal{dom}(f(T))$;
		\item[(iv)] $\textnormal{Id}(T)=T$, where $\textnormal{Id}$ is the function $t\mapsto t$.
	\end{enumerate}
\end{theorem}
	In the rest of this section, we will finish the proof of Theorem \ref{thm:regularity}. 
	
	Let $f_\mu\colon \mathbb{R}\rightarrow\mathbb{C}$ be the function $x\mapsto (x+\mu i)^{-1}$. Let $K_{f_\mu}$ denote the Schwartz kernel of the bounded operator $f_\mu(D)$. Since $K_{f_\mu}$ is pseudodifferential, it is smooth on the complement of the diagonal. Furthermore, it satisfies the following estimate:
	\begin{proposition}
		\label{prop kernel}
		There exists $C_{\mu}>0$ such that for all $x,y\in M$ with $d(x,y)\geq 1$,		$$\norm{K_{f_\mu}(x,y)}\leq C_\mu e^{-\frac{\mu}{2}d(x,y)},$$
		where $\norm{\cdot}$ denotes the fiberwise norm on $S\boxtimes S^*$.
	\end{proposition}
	To prove this, we will use the following lemma, which is an adaptation of \cite[Lemma 3.5]{Jinmin} to the bounded geometry setting.
	\begin{lemma}
	\label{lem kernel bound}
	Let $T$ be a bounded linear operator on $L^2(S)$ such that
	$$\sup_{k+j\leq\frac{3}{2}\dim M+3}\norm{D^k TD^j}_{\mathcal{B}(L^2(S))}<\infty.$$
	Then $T$ is an integral operator with a continuous Schwartz kernel $K_T(x,y)$, and there exists $C>0$, independent of $T$, such that
	$$\sup_{x,y\in M}\norm{K_T(x,y)}\leq C\cdot\sup_{k+j\leq\frac{3}{2}\dim M+3}\norm{D^k TD^j}_{\mathcal{B}(L^2(S))},$$
	where $\norm{\cdot}$ denotes the fiberwise norm on $S\boxtimes S^*$.
	\end{lemma}
	\begin{proof}
	Since $M$ has bounded Riemannian geometry, it admits an open cover $\{U_l\}_{l\in\mathbb{N}}$ and a subordinate smooth partition of unity $\{\varphi_l\}_{l\in\mathbb{N}}$ such that
	\begin{enumerate}[(i)]
	\item $0\leq\varphi_l\leq 1$ for each $l$;
	\item Each $U_l$ has compact closure;
	\item The maximum number of overlapping elements of $\{U_l\}_{l\in\mathbb{N}}$ is finite;
	\item There exists $C>0$ such that $\norm{[D,\varphi_l]}_{\mathcal{B}(L^2(S))}$ is 
	bounded above by $C$ for all $l$.
	\end{enumerate}
	See \cite[Lemmas 1.2, 1.3]{Shubin}. By repeated applications of (iv), and using the fact that $D^{2k}[D,\varphi_l](1+D^2)^{-k}$ and $D^{2k}[D,\varphi_m](1+D^2)^{-k}$ are bounded operators, we see that for every $l,m\in\mathbb{N}$ we have
	$$\sup_{k+j\leq\frac{3}{2}\dim M+3}\norm{D^k \varphi_l T\varphi_m D^j}_{\mathcal{B}(L^2(S))}<\infty.$$
	It follows from \cite[Lemma 3.5]{Jinmin} in the compact setting that $\varphi_l T\varphi_m$ is an integral operator on $M$ with continuous Schwartz kernel $K_{\varphi_l T\varphi_m}$, whence $T$ is an integral operator with continuous Schwartz kernel $K_T=\sum_{l,m}K_{\varphi_l T\varphi_m}$, which is a pointwise finite sum by (iii).
	
	For each $x\in M$, let $I_x\coloneqq\{l\in\mathbb{N}\colon x\in U_l\}$, and define the function
	$$\varphi_x\coloneqq\sum_{l\in I_x}\varphi_l.$$
	Then there is a neighbourhood of $x$ on which $\varphi_x$ takes the constant value $1$. Together with \cite[Lemma 3.5]{Jinmin}, this implies that for every $(x,y)\in M\times M$, there exists a positive constant $C_{x,y}$ such that
	$$\norm{K_T(x,y)}=\norm{\varphi_x K_T\varphi_y(x,y)}\leq C_{x,y}\cdot\sup_{k+j\leq\frac{3}{2}\dim M+3}\norm{D^k \varphi_x T\varphi_y D^j}_{\mathcal{B}(L^2(S))}.$$
	Again by repeated applications of (iv), and the fact that $D^{2k}[D,\varphi_x](1+D^2)^{-k}$ and $D^{2k}[D,\varphi_y](1+D^2)^{-k}$ are bounded operators, there exists $C'_{x,y}$ such that this is bounded above by
	$$C'_{x,y}\cdot\sup_{k+j\leq\frac{3}{2}\dim M+3}\norm{D^k T D^j}_{\mathcal{B}(L^2(S))}.$$
	Furthermore, the conditions above imply that the constants $C'_{x,y}$ are uniformly bounded above by a constant $C$, independent of $x$ and $y$, hence
	\begin{align*}
	\sup_{x,y}\norm{K_T(x,y)}&\leq C\cdot\sup_{k+j\leq\frac{3}{2}\dim M+3}\norm{D^k T D^j}_{\mathcal{B}(L^2(S))}.\qedhere
	\end{align*}
	\end{proof}
	\begin{proofof}{Proposition \ref{prop kernel}}
		Let $x,y\in M$ with $\lambda\coloneqq\textnormal{dist}(x,y)\geq 1$. Choose a smooth function $\phi\colon \mathbb{R}\rightarrow\mathbb{R}$ such that $\phi(\xi)=1$ for $|\xi|\geq 1$ and $\phi(\xi)=0$ if $|\xi|\leq\frac{1}{2}$. Let $\phi_{\lambda}(\xi)\coloneqq\phi\left(\frac{\xi}{\lambda}\right)$. Let $g_\mu$ be the function on $\mathbb{R}$ with Fourier transform
		$$\widehat{g}_\mu(\xi)=\phi_\lambda(\xi)\widehat{f}_\mu(\xi).$$
		Let $K_{g_\mu}$ denote the Schwartz kernel of $g_\mu(D)$. By Lemma \ref{lem kernel bound}, there exists $C>0$ such that for all $x,y\in M$,
		$$\norm{K_{g_\mu}(x,y)}\leq C\cdot\sup_{l\leq\frac{3}{2}\dim M+3}\norm{D^{l}g_\mu(D)}_{\mathcal{B}(L^2(S))}.$$
		For a given $l\leq\frac{3}{2}\dim M+3$, we can estimate the right-hand side as follows. Let $\psi_l$ be the function given by $\psi_l(s)=s^l g_\mu(s)$. We have 
		$$\hat{\psi}_l(\xi)=\left(\frac{1}{i}\frac{d}{d\xi}\right)^n(\phi_\lambda\hat{f})(\xi).$$ By the Fourier inversion formula
		$$g_\mu(D)=\frac{1}{2\pi}\int_{-\infty}^\infty\widehat{g}_\mu(\xi)e^{i\xi D}d\xi,$$ and the fact that $\phi_\lambda$ is supported on $|\xi|\geq\frac{\lambda}{2}$, we have:
		\begin{align*}
		\norm{D^{l}g_\mu(D)}_{\mathcal{B}(L^2(S))}&=\norm{\psi_l(D)}_{\mathcal{B}(L^2(S))}\\
		&\leq\frac{1}{2\pi}\int_{|\xi|\geq\frac{\lambda}{2}}\left|\widehat{\psi}_l(\xi)\right|\,d\xi\\
		&\leq C_1\sum_{j=0}^l\int_{|\xi|\geq\frac{\lambda}{2}}\left|\widehat{f}^{(l-j)}_\mu(\xi)\right|\,d\xi\\
		&\leq C_2\int_{|\xi|\geq\frac{\lambda}{2}}e^{-\mu\xi}\mathbbm{1}_{(0,\infty)}\,d\xi\\
		&\leq C_3 e^{-\frac{\mu\lambda}{2}},
		\end{align*}
		for some $C_1, C_2, C_3>0$, and where we have used that $\widehat{f}_\mu(\xi)=\frac{2\pi\mu}{i}e^{-\mu\xi}\mathbbm{1}_{(0,\infty)}$, with $\mathbbm{1}$ being the indicator function. It follows that for all $x, y$ with $d(x,y)\geq 1$, there exists $C_\mu>0$ such that
		$\norm{K_{g_\mu}(x,y)}\leq C_\mu e^{-\frac{\mu}{2}d(x,y)}$
		for a constant $C_\mu$. Now a standard finite-propagation argument for the wave operator shows that $K_{f_\mu}(x,y)=K_{g_\mu}(x,y)$ whenever $d(x,y)\geq 1$, and we conclude.
	\end{proofof}
	\begin{corollary}
		\label{cor:kernel}
		For any $k\in\mathcal{S}_u^\Gamma$, we have
		$$\norm{(D+\mu i)^{-1}k(x,y)}\leq C_{\mu}e^{-\frac{\mu}{2}d(x,y)},$$
		where $C_\mu$ is a constant depending on $\mu$ and $\norm{\cdot}$ is the fiberwise norm on $S\boxtimes S^*$.
	\end{corollary}
	\begin{proof}
		For any $y\in M$, set
		$$L_y=\{z\in M\,|\,k(z,y)\neq 0\}.$$
		Observe that $\sup_{y\in M}\{\textnormal{diam}(L_y)\}\leq\textnormal{prop}(k)$. By bounded Riemannian geometry, we have
		$\sup_{y\in M}\{\textnormal{vol}(L_y)\}\leq C'_k$
		for some constant $C'_k$. Let $B_1$ be the set of $(x,y)\in M\times M$ with $d(x,y)<1$, and let $B_1^c$ denote its complement. Let $K_{<1}\coloneqq K_{f_\mu}\mathbbm{1}_{B_1}$ and $K_{\geq 1}\coloneqq K_{f_\mu}\mathbbm{1}_{B_1^c}.$ Thus $K_{f_\mu}=K_{<1}+K_{\geq 1}$. We have, for all $x,y\in M$,
		\begin{align*}
		\norm{\int_{M}K_{\geq 1}(x,z)k(z,y)\,dz}&\leq\int_{L_y}\norm{K_{\geq 1}(x,z)k(z,y)}\,dz\\
		&\leq C_\mu\int_{L_y} e^{-\frac{\mu}{2}(d(x,y)-d(z,y))}\norm{k(z,y)}\,dz\\
		&\leq C_\mu e^{-\frac{\mu}{2}d(x,y)}\int_{L_y} e^{\frac{\mu}{2}d(z,y)}\norm{k(z,y)}\,dz\\
		&\leq C_\mu e^{-\frac{\mu}{2}d(x,y)}\cdot C'_k e^{\frac{\mu}{2}\textnormal{prop}(k)}\norm{k}_\infty\\
		&=C_{k}''e^{-\frac{\mu}{2}d(x,y)},
		\end{align*} 
		for a new constant $C_{k}''$. Now, we have 
		\begin{align*}
		\norm{\int_M K_{<1}(x,z)k(z,y)\,dz}&\leq\norm{(D+\mu i)^{-1}k(x,y)}+\norm{\int_M K_{\geq 1}(x,z)k(z,y)\,dz}\\
		&\leq C+C_k'',
		\end{align*}
		for a constant $C>0$ given by the Sobolev embedding theorem and uniform boundedness of $M$. Also note that the function
		$$(x,y)\mapsto \int_M K_{<1}(x,z)k(z,y)\,dz$$
		is supported on some ball $B$ of finite radius around the diagonal, so there exists $C_{k}>0$ such that
		\begin{align*}
		\norm{(D+\mu i)^{-1}k(x,y)}&\leq\norm{\int_M K_{\geq 1}(x,z)k(z,y)\,dz}+(C+C_k'')\mathbbm{1}_B(x,y)\\
		&\leq C_{k}e^{-\frac{\mu}{2}d(x,y)}.\qedhere\\
		\end{align*}
	\end{proof}
	We now prove the key technical result of this section, Theorem \ref{thm:regularity}.\\
	
	\begin{proofof}{Theorem \ref{thm:regularity}}
		We will argue in the case of the operator $D+\mu i$, with the case of $D-\mu i$ being entirely analogous. Let $k\in\mathcal{S}_u^\Gamma$. Fix a countable open cover $M$ whose elements have uniformly bounded diameter. Let $\{\rho_j\}_{j\in\mathbb{N}}$ be a smooth partition of unity subordinate to this cover. We can write
		$$k(x,y)=\sum_{j}\rho_j(x)k(x,y).$$
		Note that the Schwartz kernel $k_j\coloneqq(D+\mu i)^{-1}(\rho_j k)$ is smooth but may not be compactly supported in $M\times M$. However, since it is compactly supported in the second coordinate, the sum
		$$\tilde{k}(x,y)\coloneqq\sum_j k_j(x,y)=(D+\mu i)^{-1}k(x,y)$$
		still makes sense at each point $(x,y)\in M\times M$ and moreover is $\Gamma$-invariant.

Since $M\times M$ is complete, there exists a family $\{f_{\varepsilon}\}_{\varepsilon \in (0,1]}$ of smooth functions on $M \times M$, invariant under the diagonal action of $\Gamma$, such that for all $\varepsilon \in (0,1]$, we have
\begin{enumerate}[(i)]
\item $f_{\varepsilon} = 1$ on $B_{\frac{1}{\varepsilon}}(\supp(k))$;
\item $f_{\varepsilon} = 0$ outside $B_{\frac{3}{\varepsilon}}(\supp(k))$;
\item $\|df_{\varepsilon}\|_{\infty} \leq \varepsilon$.
\end{enumerate}

Write $d_1 f_{\varepsilon} \in C^{\infty}(M \times M, T^*M \times M)$ for the derivative of $f_{\varepsilon}$ in the first factor. Clifford multiplication then defines a map $c(d_1 f_{\varepsilon})\colon M\times M\to\End(S)$. Note that by property (ii) above, $f_\varepsilon\tilde k$ and $c(d_1 f_{\varepsilon})\tilde k$ belong to $\mathcal{S}_u^\Gamma$ for each $\varepsilon \in (0,1]$. We have: 
\begin{equation}
\label{eq commutator}
        (D+\mu i)(f_\varepsilon\tilde k)=f_{\varepsilon}(D+i\mu)\tilde k+c(d_1 f_{\varepsilon})\tilde k
        =k+c(d_1 f_{\varepsilon})\tilde k.
\end{equation}
Here the composition $c(d_1 f_{\varepsilon})\tilde k$ is given by
\beq{eq def sigma kappa}
(c(d_1 f_{\varepsilon})\tilde k)(x,y) = c(d_1 f_{\varepsilon})(x,y)_{x} \circ\tilde k(x,y),
\eeq
where $c(d_1 f_{\varepsilon})(x,y)_{x}$ denotes the value of $c(d_1 f_{\varepsilon})(x,y)\in\End(S)$ at the fiber over $x$.

Let $\phi\colon \mathbb{C}[M]^\Gamma\rightarrow\mathcal{B}(H')$ be a $*$-representation. 
By \eqref{eq commutator},
we have 
\begin{equation}
\label{eq norm commutator}
\|\phi(((D+\mu i)(f_\varepsilon\tilde k) - k)\oplus 0)\|_{\mathcal{B}(H')} = 
\|\phi(c(d_1 f_{\varepsilon})\tilde k\oplus 0)\|_{\mathcal{B}(H')}.
\end{equation}
We now show that for $\mu$ sufficiently large, for any $k\in\mathcal{S}_u^\Gamma$ the quantity \eqref{eq norm commutator} approaches $0$ in the limit $\varepsilon\to 0$, at a rate independent of the representation $\phi$.

To do this, let $\Delta_{M\times M}$ denote the diagonal in $M\times M$, and consider the open cover $\mathcal{V}\coloneqq\{V_l\}_{l\in\mathbb{Z}_{\geq 0}}$ of $M\times M$ by the sets		$$
		V_l\coloneqq
		\begin{cases}
		B_{\frac{1}{2}}(\Delta_{M\times M})&\textnormal{ if $l=0$,}\\
		B_{l+\frac{1}{2}}(\Delta_{M\times M})\backslash\overline{B_{l-\frac{3}{4}}(\Delta_{M\times M})}&\textnormal{ otherwise.}
\end{cases}
$$
		Note that each $V_l$ is preserved by the diagonal action of $\Gamma$ on $M\times M$.
		Let $\{\psi_l\}_{l\in\mathbb{Z}\geq 1}$ be a smooth $\Gamma$-invariant partition of unity subordinate to $\mathcal{V}$. For each $l$, the kernel 
		$$\tilde{k}_l(x,y)\coloneqq\psi_l\tilde{k}(x,y)$$
		is an element of $\mathcal{S}_u^\Gamma$ with propagation at most $\sqrt{2}(l+\frac{1}{2})$. 
For each $\varepsilon\in(0,1]$, there exists a finite $N_{\varepsilon}$ such that 
$$c(d_1 f_{\varepsilon})\tilde k=\sum_{l=1}^{N_{\varepsilon}}c(d_1 f_{\varepsilon})\tilde{k}_l.$$
By Lemmas \ref{lem:keyestimate} and \ref{lem:gromovbishop}, there exist constants $C_1$, $C_2$, and $C_3$, independent of $\phi$, $k$, and $\varepsilon$, such that
		\begin{align*}
		\norm{\phi\big(c(d_1 f_{\varepsilon})\tilde{k}_l\oplus 0\big)}_{\mathcal{B}(H')}&\leq C_1 V^4_{\sqrt{2}(l+\frac{1}{2})}\|df_{\varepsilon}\|_{\infty}\cdot\sup_{w,z\in\tilde{Z}}\|(\tilde k_l)_{wz}\|_{\mathcal{B}(\mathcal{H}_M)}\\
		&\leq C_2 e^{C_3l}\|df_{\varepsilon}\|_{\infty}\cdot\sup_{w,z\in\tilde{Z}}\|(\tilde k_l)_{wz}\|_{\mathcal{B}(\mathcal{H}_M)}.
		\end{align*}
Here $(\tilde k_l)_{wz}$ denotes the $(w,z)$-entry of $\tilde k_l$ considered as a $\tilde Z\times\tilde Z$-matrix, as in subsection \ref{subsubsec:discretizing}, and we have used that the map $\oplus\, 0$ from \eqref{eq oplus 0} preserves propagation and $L^2$-norm. The entrywise norm $\|(\tilde k_l)_{wz}\|_{\mathcal{B}(\mathcal{H}_M)}$ is bounded above by $\normbig{\tilde{k}_l}_\infty$ times a constant that is independent of $w,z\in\tilde Z$, and hence
		\begin{align*}
		\norm{\phi\big(c(d_1 f_{\varepsilon})\tilde{k}_l\oplus 0\big)}_{\mathcal{B}(H')}\leq C_4 e^{C_3l}\|df_{\varepsilon}\|_{\infty}\normbig{\tilde{k}_l}_\infty,
		\end{align*}
where $C_4$ is again independent of $k$. Since $\tilde{k}_l$ is supported away from $B_{l-\frac{3}{4}}(\Delta_{M\times M})$, Proposition \ref{prop kernel} implies that there exist constants $C_\mu$ (depending only on $\mu$) and $C_k$ (depending only on the initial kernel $k$) such that
		$$\normbig{\tilde{k}_l}_\infty\leq C_\mu C_k e^{-\frac{\mu l}{2}}.$$ 
This implies that
\begin{align*}
\|\phi\big(c(d_1 f_{\varepsilon})\tilde k\oplus 0\big)\|_{\mathcal{B}(H')}&=\norm{\phi\Big(\sum_{l=0}^{N_{\varepsilon}}c(d_1 f_{\varepsilon})\tilde{k}_l\oplus 0\Big)}_{\mathcal{B}(H')}\\
&\leq\sum_{l=0}^{N_{\varepsilon}}\norm{\phi\big(c(d_1 f_{\varepsilon})\tilde{k}_l\oplus 0\big)}_{\mathcal{B}(H')}\\
&\leq C_4 C_\mu C_k\sum_{l=0}^{N_{\varepsilon}}\|df_{\varepsilon}\|_{\infty}\cdot e^{C_3l}e^{-\frac{\mu l}{2}}.
\end{align*}
For $\mu$ sufficiently large, the last sum converges and is bounded above by $$C'C_k\norm{df_{\varepsilon}}\leq C'C_k\varepsilon,$$ 
for some constant $C'$ independent of $\varepsilon$ and $k$. Thus for $\mu$ large enough, \eqref{eq norm commutator} approaches $0$ in the limit $\varepsilon\to 0$, independent of the initial choice of $k$.

This shows that any $k\in\mathcal{S}_u^\Gamma$ can be approximated arbitrarily closely in the norm of $C^*_{\textnormal{max},u}(M;L^2(S))^\Gamma$ by elements in the image of $D+\mu i$. It follows that, since $\mathcal{S}_u^\Gamma$ is dense in $C^*_{\textnormal{max},u}(M;L^2(S))^\Gamma$, the Hilbert module operator $D+\mu i$ has dense range.
\end{proofof}
	
	\hfill\vskip 0.3in
	\section{Proof of the main theorem}
	\label{sec:maxuniform}
	In this section, we first define the \emph{maximal uniform index} of $D$ and show that under the positive scalar curvature assumption, this index vanishes. We then use this to prove our main result, Theorem \ref{thm:main}.
	\subsection{The maximal uniform index}
	\hfill\vskip 0.05in
	\noindent Let $\mathcal{M}\coloneqq\mathcal{M}(C^*_{\textnormal{max},u}(M;L^2(S))^\Gamma)$ denote the multiplier algebra of $C^*_{\textnormal{max},u}(M;L^2(S))^\Gamma$.	
	The short-exact sequence of $C^*$-algebras
	$$0\rightarrow C^*_{\textnormal{max},u}(M;L^2(S))^\Gamma\rightarrow \mathcal{M}\rightarrow \mathcal{M}/C^*_{\textnormal{max},u}(M;L^2(S))^\Gamma\rightarrow 0$$
	induces a six-term exact sequence in $K$-theory:
	$$\xymatrix{ K_0(C^*_{\textnormal{max},u}(M;L^2(S))^\Gamma) \ar[r] & K_0(\mathcal{M}) \ar[r] & K_0(\mathcal{M}/C^*_{\text{max},u}(M;L^2(S))^\Gamma) \ar[d]^{\partial} & \\
		K_1(\mathcal{M}/C^*_{\text{max},u}(M;L^2(S))^\Gamma) \ar[u]^{\partial}& K_1(\mathcal{M}) \ar[l]& K_1(C^*_{\textnormal{max},u}(M;L^2(S))^\Gamma). \ar[l] }$$
	Let $\chi\colon \mathbb{R}\rightarrow\mathbb{R}$ be a continuous odd function with limit $1$ at $\infty$ (called a \emph{normalizing function}). We now apply the functional calculus of Theorem \ref{thm:functionalcalculus} with $\mathcal{N}=C^*_{\textnormal{max},u}(M;L^2(S))^\Gamma$ and $T=D$ to form the bounded adjointable operator 
	$$\chi(D)\colon C^*_{\textnormal{max},u}(M;L^2(S))^\Gamma\rightarrow C^*_{\textnormal{max},u}(M;L^2(S))^\Gamma.$$
	In the remainder of this section we prove:
	\begin{proposition}
		\label{prop:fredholmness}
		The operator $\chi(D)$ is invertible modulo $C^*_{\textnormal{max},u}(M;L^2(S))^\Gamma$ and so defines a class
		\begin{equation}
		\label{eq class}	
		[\chi(D)]\in K_{n+1}(\mathcal{M}/C^*_{\textnormal{max},u}(M;L^2(S))^\Gamma)
		\end{equation}
		that is independent of the choice of normalizing function $\chi$.
	\end{proposition}
	\begin{definition}
		\label{def:max index}
		The \emph{maximal uniform index} of $D$, denoted $\ind_{\textnormal{max},u}(D)$, is the image of $[\chi(D)]$ under the boundary map
		$$\partial\colon K_{n+1}(\mathcal{M}/C^*_{\text{max},u}(M;L^2(S))^\Gamma)\rightarrow K_{n}(C^*_{\text{max},u}(M;L^2(S))^\Gamma).$$
	\end{definition}
	\begin{proofof}{Proposition \ref{prop:fredholmness}}
		Without loss of generality, let us work in the case when $n$ is even. By Theorem \ref{thm:regularity}, we have $\chi(D)\in\mathcal{M}(C^*_{\textnormal{max},u}(M;L^2(S))^\Gamma).$ To see that $\chi(D)$ defines a class in $K_{1}(\mathcal{M}/C^*_{\textnormal{max},u}(M;L^2(S))^\Gamma),$ it suffices to show that for any $f\in C_0(\mathbb{R})$, we have $f(D)\in C^*_{\textnormal{max},u}(M;L^2(S))^\Gamma$, since $\chi^2-1\in C_0(\mathbb{R})$. Since $M$ has bounded Riemannian geometry, \cite[Proposition 2.10]{Roe} implies that for any $f\in\mathcal{S}(\mathbb{R})$ with compactly supported Fourier transform, the operator $f(D)$ is given by a smooth Schwarz kernel that is uniformly bounded along with all of its derivatives. The fact that $\mathcal{S}(\mathbb{R})$ is dense in $C_0(\mathbb{R})$, together continuity of the functional calculus homomorphism (part (i) of Theorem \ref{thm:functionalcalculus}) shows this is true for general $f\in C_0(\mathbb{R})$. Finally, since the difference of any two normalizing functions lies in $C_0(\mathbb{R})$, the class $[\chi(D)]$ is independent of the choice of $\chi$.
	\end{proofof}
	
	\subsection{Vanishing of the maximal uniform index}
	\hfill\vskip 0.05in
	\noindent Suppose that the scalar curvature function $\kappa$ of $(M^n,g)$ is uniformly positive; that is, there exists $c>0$ such that $\kappa\geq c$. Recall that the Lichnerowicz formula states that
	\begin{equation}
	\label{eq Lichnerowicz}	
	D^2v=\nabla^*\nabla v+\frac{\kappa}{4} v
	\end{equation}
	for all elements $v\in \mathcal{S}_u^\Gamma$, where $\nabla$ is the lift of the Levi-Civita connection to $S$. 
	\begin{proposition}
	The spectrum of the Hilbert $C^*_{\textnormal{max},u}(M;L^2(S))^\Gamma$-module operator $D$ has a gap $(-\frac{\sqrt{c}}{2},\frac{\sqrt{c}}{2})$.
	\end{proposition}
	\begin{proof}
	To be more precise, let us write $\overline{D}$ for the  closure of the operator $D$ on $\mathcal{S}_u^\Gamma$, as in subsection \ref{subsection hilbert operator}. By Corollary $\ref{cor:selfadj}$, the operator $\overline{D}$ is self-adjoint.  We wish to show that
	\begin{equation}
	\label{eq Dbar squared}
	\langle\overline{D} v, \overline{D} v\rangle\geq\frac{c}{4}\langle v,v\rangle
	\end{equation}
	for all $v$ in the domain of $\overline{D}$, where $\langle\,\,,\,\rangle$ denotes the $C^*_{\textnormal{max},u}(M;L^2(S))^\Gamma$-valued inner product. In fact, it suffices to show that 
	\[  \langle\overline{D} w, \overline{D} w\rangle\geq\frac{c}{4}\langle w,w\rangle \]
	for all $w\in \mathcal{S}_u^\Gamma $. Indeed, for any $v$ in the domain of $\overline{D}$, there is a sequence of elements $(w_n)_{n\in \mathbb N}$ in $\mathcal{S}_u^\Gamma$ such that  $v=\lim_{n\to \infty}w_n$ and $\overline{D} v = \lim_{n\to \infty}\overline{D} w_n$. Therefore, if we have shown
	\[ \langle \overline{D} w_n, \overline{D} w_n\rangle \geq  \frac{c}{4}\langle w_n,w_n\rangle \textup{ for all } n\in \mathbb N,  \]
	it would then follow that
	\[ \langle\overline{D} v, \overline{D} v\rangle\geq\frac{c}{4}\langle v,v\rangle.  \]
	 Now for any $w\in\mathcal{S}_u^\Gamma$, we have 
	$$\langle\nabla^*\nabla w,w\rangle=(\nabla^*\nabla w)^*w = (\nabla w)^*(\nabla w)\geq 0 \textup{ in } C^*_{\textnormal{max},u}(M;L^2(S))^\Gamma.$$
	It then follows from \eqref{eq Lichnerowicz} that
\begin{equation*}
\langle\overline{D} w, \overline{D} w\rangle = \langle D^2 w, w\rangle\geq\frac{c}{4}\langle w,  w\rangle
\end{equation*}
in  $C^*_{\textnormal{max},u}(M;L^2(S))^\Gamma$.
This finishes the proof. 
%
	\end{proof}	
	To see that this implies vanishing of $\ind_{\textnormal{max,u}}(D)$, 	
	consider the function $\chi$ on $\mathbb{R}\backslash\{0\}$ given by
	$$\chi(x)=\frac{x}{|x|}.$$ 
	Since $D$ has a spectral gap at $0$, we may use the functional calculus for regular operators on Hilbert modules to form the class $[\chi(D)]$ in \eqref{eq class}. This class lifts directly to $K_{n+1}(\mathcal{M})$, and hence its image under the boundary map vanishes in $K_n(C^*_{\textnormal{max},u}(M;L^2(S))^\Gamma)$.
	
	This yields a version of Theorem \ref{thm:main} for the uniform maximal Roe algebra:
	\begin{theorem}
		\label{thm:vanishing}
		Let $(M^n,g)$ be a complete $\Gamma$-spin Riemannian manifold with bounded Riemannian geometry. Let $\Gamma$ be a countable discrete group acting and properly and isometrically on $M$, satisfying Assumption \ref{ass:condition}. If $M$ has uniformly positive scalar curvature, then
		 $$\ind_{\textnormal{max},u}(D)=0\in K_n(C^*_{\textnormal{max},u}(M;L^2(S))^\Gamma).$$
	\end{theorem}
	\subsection{Vanishing of the maximal index}
	\hfill\vskip 0.05in
	\noindent We can now complete the proof of our main result, Theorem \ref{thm:main}. Let us first recall the definition of the maximal higher index of $D$ \cite[4.14]{GWY}. We will work in the case when the dimension $n$ is even, with the odd case being analogous.
	
	Given a normalizing function $\chi$, we can form the operator $\chi(D)$ using the functional calculus for self-adjoint operators on $L^2(S)$. Pick a locally finite $\Gamma$-invariant open cover $\{U_i\}_{i\in\mathbb{N}}$ of $M$ with the property that
	$$\sup_{i\in\mathbb{N}}\{\textnormal{diam}(U_i/\Gamma)\}<C$$
	for some $C>0$. Let $\{\phi_i\}_{i\in\mathbb{N}}$ be a continuous partition of unity subordinate to $\{U_i\}_{i\in\mathbb{N}}$. Then the sum
	$$F_D\coloneqq\sum_i\phi_i^{\frac{1}{2}}\chi(D)\phi_i^{\frac{1}{2}}$$
	defines a bounded, $\Gamma$-invariant, locally compact operator on $L^2(S)$ with finite propagation.	
Consider the matrix of bounded operators
	$$W_D=\begin{pmatrix}1&F_D\\0&1\end{pmatrix}\begin{pmatrix}1&0\\-F_D^*&1\end{pmatrix}\begin{pmatrix}1&F_D\\0&1\end{pmatrix}\begin{pmatrix}0&-1\\1&0\end{pmatrix}.$$
	Each entry of $W_D$ has finite propagation, and one verifies that
	$$P_D\coloneqq W_D\begin{pmatrix}1&0\\0&0\end{pmatrix}W_D^{-1}-\begin{pmatrix}1&0\\0&0\end{pmatrix}$$
	is a projection in $M_2(\mathbb{C}[M;L^2(S)]^\Gamma)$. We define the \emph{maximal higher index of $D$ on $L^2(S)$} to be the $K$-theoretic class of $P_D$:
	$$\ind^{L^2(S)}_{\textnormal{max}}(D)\coloneqq[P_D]\in K_0(C^*_{\textnormal{max}}(M;L^2(S))^\Gamma).$$
	Now the map $\oplus\,0\colon \mathbb{C}[M;L^2(S)]^\Gamma\rightarrow\mathbb{C}[M]^\Gamma$ 
	extends to an injective $*$-homomorphism between the maximal completions of both sides that we will still denote by
	$$\oplus\,0\colon C^*_{\textnormal{max}}(M;L^2(S))^\Gamma\rightarrow C^*_{\textnormal{max}}(M)^\Gamma.$$
	This induces a homomorphism on $K$-theory:
	$$\oplus\,0_{\,*}\colon K_0(C^*_{\textnormal{max}}(M;L^2(S))^\Gamma)\rightarrow K_0(C^*_{\textnormal{max}}(M)^\Gamma).$$ 
	The \emph{maximal higher index of $D$} is the image of $\ind^{L^2(S)}_{\textnormal{max}}(D)$ under this map: 
	$$\ind_{\textnormal{max}}(D)\coloneqq\oplus\,0_{\,*}(\ind^{L^2(S)}_{\textnormal{max}}(D))\in K_0(C^*_{\textnormal{max}}(M)^\Gamma).$$

		
	Equivalently, the elements $\ind^{L^2(S)}_{\textnormal{max}}(D)$ and $\ind_{\textnormal{max}}(D)$ can be obtained in the following way. Let $\chi'$ be a normalizing function with compactly supported Fourier transform. 
	Then $\chi'(D)$ has finite propagation, and the matrix
	$$W'_D=\begin{pmatrix}1&\chi'(D)\\0&1\end{pmatrix}\begin{pmatrix}1&0\\-\chi'(D)^*&1\end{pmatrix}\begin{pmatrix}1&\chi'(D)\\0&1\end{pmatrix}\begin{pmatrix}0&-1\\1&0\end{pmatrix}$$
	defines a projection
	$$P'_D=W'_D\begin{pmatrix}1&0\\0&0\end{pmatrix}(W'_D)^{-1}-\begin{pmatrix}1&0\\0&0\end{pmatrix}$$
	in $M_2(\mathbb{C}[M;L^2(S)])^\Gamma$. One then verifies that 
	$$\ind^{L^2(S)}_{\textnormal{max}}(D)=[P_D]=[P'_D],$$
	whence they give rise to the same element $\ind_{\textnormal{max}}(D)\in K_0(C^*_{\textnormal{max}}(M)^\Gamma)$.
	\begin{proof}[Proof of Theorem \ref{thm:main}]
		Assume, without loss of generality, that $n$ is even. By definition of the maximal equivariant uniform Roe algebra on $L^2(S)$, there is a natural inclusion $\iota\colon C^*_{\textnormal{max},u}(M;L^2(S))^\Gamma\rightarrow C^*_{\textnormal{max}}(M;L^2(S))^\Gamma.$
		The composition
		$$C^*_{\textnormal{max},u}(M;L^2(S))^\Gamma\xrightarrow{\iota}C^*_{\textnormal{max}}(M;L^2(S))^\Gamma\xrightarrow{\oplus\,0} C^*_{\textnormal{max}}(M)^\Gamma$$
		induces a composition of group homomorphisms
		$$K_0(C^*_{\textnormal{max},u}(M;L^2(S))^\Gamma)\xrightarrow{\iota_*}K_0(C^*_{\textnormal{max}}(M;L^2(S))^\Gamma)\xrightarrow{\oplus\,0\,_*}K_0(C^*_{\textnormal{max}}(M)^\Gamma).$$
		Choose a normalizing function $\chi'$ with compactly supported Fourier transform. Then $\chi'(D)$ has finite propagation, and its associated projection $P'_D$ represents $\ind_{\textnormal{max,u}}(D)$. By Theorem \ref{thm:vanishing}, uniform positive scalar curvature implies that 
		$$\ind_{\textnormal{max,u}}(D)=0.$$
		Consequently, we have 
		\begin{align*}
		\ind_{\textnormal{max}}(D)&=\oplus\,0\,_*\circ\iota_*(\ind_{\textnormal{max,u}}(D))=0.\qedhere
		\end{align*}

	\end{proof}
	\hfill\vskip 0.3in
	\bibliographystyle{abbrv}
	\bibliography{maximal.bib}
\end{document}